\documentclass[12pt, a4paper, parskip=half, abstracton]{scrartcl}

\usepackage{array}
\usepackage{marginnote}
\usepackage{xcolor}
\usepackage{amscd,amssymb,amsfonts,amsmath,latexsym,amsthm}
\usepackage[all,cmtip]{xy}
\usepackage{mathrsfs}
\usepackage{graphicx}
\usepackage{bm} 

\usepackage{etoolbox}

\def\hB{\hspace*{\fill}$\qed$}

\usepackage[nottoc]{tocbibind} 
\usepackage{thmtools}
\usepackage{defs_pp1}
\usepackage{slashed}
\usepackage[utf8]{inputenc}
\usepackage{microtype}
\usepackage[english]{babel}
\usepackage{mathtools}
\usepackage{esvect}

\usepackage{hyperref}

\usepackage[nameinlink, capitalise]{cleveref}

\title{Products in $KK$- and $E$-theory }
\author{
Ulrich Bunke\thanks{Fakult{\"a}t f{\"u}r Mathematik,
Universit{\"a}t Regensburg,
93040 Regensburg,
ulrich.bunke@mathematik.uni-regensburg.de}  and 
Benjamin Dünzinger  \thanks{Fakult{\"a}t f{\"u}r Mathematik,
Universit{\"a}t Regensburg,
93040 Regensburg,
benjamin\-.duenzinger\-@mathe\-matik\-.uni-regensburg.de}
}

\numberwithin{equation}{section}
\newtheorem{theorem}{Theorem}[section] 
\newtheorem{prop}[theorem]{Proposition}
\newtheorem{lem}[theorem]{Lemma}

\newtheorem{ddd}[theorem]{Definition}
\newtheorem{kor}[theorem]{Corollary}

\theoremstyle{remark}
\theoremstyle{definition}

\newtheorem{ex}[theorem]{Example}
\newtheorem{rem}[theorem]{Remark}

\newcommand{\ee}{\mathrm{e}}

\newcommand{\fin}{\mathrm{fin}}

\newcommand{\EE}{\mathrm{E}}
\newcommand{\sepa}{\mathrm{sep}}

\newcommand{\nCalg}{C^{*}\mathbf{Alg}^{\mathrm{nu}}}

\newcommand{\Fib}{{\mathrm{Fib}}}

\newcommand{\cK}{\mathcal{K}}

\newcommand{\bA}{{\mathbf{A}}}

 \newcommand{\Cat}{{\mathbf{Cat}}}

\renewcommand{\hom}{\mathtt{hom}}

\renewcommand{\Pr}{\mathbf{Pr}}

\newcommand{\perf}{\mathrm{perf}}

\newcommand{\kk}{\mathrm{kk}}
\newcommand{\KK}{\mathrm{KK}}

\newcommand{\st}{\mathrm{st}}

 \newcommand{\homol}{\mathrm{homol}}
\newcommand{\lex}{\mathrm{lex}}

 \renewcommand{\Mod}{\mathrm{Mod}}
 
 \newcommand{\Nuc}{\mathbf{Nuc}}
 \newcommand{\bbK}{\mathbb{K}}
 \newcommand{\restr}{\mathrm{restr}}

\begin{document}  \maketitle 
  
   \begin{abstract} 
   {In this  note, we give an explicit description of countable products in $KK$- and $E$-theory and  provide several applications.}
   \end{abstract} 
  \tableofcontents

  \section{Introduction}
  
  The classical bivariant abelian group-valued $KK$-theory functor for separable {$C^{*}$-}algebras  was introduced in   \cite{kasparovinvent}; an alternative picture was provided in  \cite{MR899916}. 
  In  \cite{higsondiss} {and}   \cite{MR2193334} $KK$-theory   was interpreted and  characterized as
  a functor taking values in an    additive or a triangulated category, respectively. 
  The 
  additive category-valued 
  $E$-theory functor for separable $C^{*}$-algebras was introduced in  \cite{MR1068250}. Alternative models were given in  \cite{zbMATH04182148} and \cite{Guentner_2000}.
 {Recently, 
 these classical  $E$- and $KK$-theory functors  were refined to  functors   with values in stable $\infty$-categories   \cite{LN}, \cite{KKG}, \cite{bunke:2023aa}, \cite{budu}, \cite{Datta:2025aa}.} 

{In the present paper, we propose a new definition of $E$- and $KK$-theory. To this end, in   \cref{koptrhertgertget},  we introduce  the notions of $E$- and $KK$-homological functors, and in 
  \cref{hrerhrgterhet}  we define $E$- or $KK$-theory equivalences as morphisms between $C^{*}$-algebras that are inverted by all  $E$- or $KK$-homological functors, respectively.} 
 {\begin{ddd}[\cref{hrteoptrhertg2}]\label{okhperthtrege} We define the $E$- and $KK$-theory functors
  \begin{equation}\label{vdfsvsdfvsdfvsdf}\ee:\nCalg\to \EE\ , \quad \kk:\nCalg\to \KK
\end{equation}  as the    Dwyer-Kan localizations that  invert
  all $E$-   or $KK$-theory equivalences, respectively.  \end{ddd} }
 {The precise relation with previous definitions will be explained in \cref{jtzrhzrthrh} below.}
  
 {Besides the obvious universal properties formulated in \cref{zjpoztjzhrtzh}, we show in \cref{erherzhgrt9} that
  the functors in \eqref{vdfsvsdfvsdfvsdf} are themselfes $E$- or $KK$-homological, and that they satisfy the universal property \begin{align*}
 \ee^{*}&:\Fun^{\colim}(\EE,\cC)\stackrel{\simeq}{\to} \Fun^{E-\homol}(\nCalg,\cC)\ ,\\
  \kk^{*}&:\Fun^{\colim}(\KK,\cC)\stackrel{\simeq}{\to} \Fun^{KK-\homol}(\nCalg,\cC)\end{align*}
  for any cocmplete stable $\infty$-category $\cC$; see \cref{iogjoegwerfwerfwerf}. 
  In view of \cref{okpherthertgertg}, after going to homotopy categories and restricting to separable $C^{*}$-algebras, the
  functors defined in \cref{okhperthtrege} become equivalent to the classical, triangulated category-valued $E$- or $KK$-theory functors.}
  
      As a consequence of {the definition of $\ee$ and $\kk$ as Dwyer-Kan localizations}, every object in the categor{ies} $\EE$ or $\KK$ is represented by some $C^{*}$-algebra. 
Previously, it was only  known that  $\aleph_{1}$-compact objects  can be represented by separable $C^{*}$-algebras.
Since the Dwyer-Kan localizations {in} \cref{okhperthtrege} are governed by   calculi of fractions, we can further conclude that also every morphism in $\EE$ or $\KK$   is represented by some morphism between $C^{*}$-algebras.

Infinite products of non-trivial $C^{*}$-algebras are never separable. Taking advantage of the fact that
the $E$- and $KK$-theory functors are now defined on all $C^{*}$-algebras, we can study the question
to which extend the   $E$- and $KK$-theory functors preserve infinite products.
We show by example, that $\ee$ and $\kk$ do not preserve all  infinite products.
In contrast, note that $\ee$ and $\kk$ preserve arbitrary coproducts; see \cref{jigowrefwerfgw}.
On the positive side,  an immediate consequence of \cref{okphjzhzjrtjzthr} is:

\begin{theorem}\label{hrrthregrt}
If $(A_{i})_{i\in I}$ is a countable family of $\bbK$-stable $C^{*}$-algebras, then the comparison maps
$$\ee(\prod_{i\in I}A_{i})\to \prod_{i\in I}\ee(A_{i})\ , \qquad  \kk(\prod_{i\in I}A_{i})\to \prod_{i\in I}\kk(A_{i}) $$
are equivalences.
\end{theorem}
Recall{,} that $A$ in $\nCalg$ is $\bbK$-stable if there exists an isomorphism $A\otimes \bbK\cong A$.

We have  right Bousfield localizations \begin{equation}\label{gwerfwerrewgwergerwgwe}
b^{E}:{\Mod_{KU}(\Sp)}\rightleftarrows \EE:K^{E}\ , \qquad    b^{KK}:{\Mod_{KU}(\Sp)}\rightleftarrows \KK:K^{KK}\ ,
\end{equation}
where $K^{E}$ and $K^{KK}$ are the functors corepresented by $\ee(\C)$ or $\kk(\C)$, respectively.
The essential images of $b^{E}$ and $b^{KK}$ are called the bootstrap classes in $\EE$ and $\KK$.
They are  the localizing subcategories  generated by $\ee(\C)$ or $\kk(\C)$, respectively. It is known that the bootstrap classes do not exhaust $\EE$ or $\KK$ (this follows, e.g., from \cite[Prop. 3.85 (1) and (2)]{budu}). This fact raises the question of understanding the Verdier quotients
$$\EE/\Mod_{KU}(\Sp)\ , \qquad \KK/\Mod_{KU}(\Sp)\ .$$ 
By identifying the bootstrap classes with  $\Mod_{KU}(\Sp)$, we can consider them as understood.  The following
result then allows to construct interesting objects in these quotients by forming infinite products of objects in the bootstrap classes.
\begin{theorem}[\cref{hopkethretgertg}.\ref{hopkethretgertg1}]
The $E$- and $KK$-theoretic bootstrap classes are not closed under forming infinite products.
\end{theorem}
More concretely, we will show in \cref{okhpertertger} that 
 $\prod_{\nat} \ee(\C)$ and
$\prod_{\nat} \kk(\C)$ do not belong to the respective bootstrap class.
To this end, we use the symmetric monoidal structure associated with the maximal tensor product $\otimes_{\max}$ and the failure of the Künneth formula. 
{The constructions going into the proof of \cref{okhpertertger}  lead to   \cref{opjwefgioherjgkljsdfg} asserting  that the functor
 $$K^{E}:=\map_{\EE}(\ee(\C),-):\EE\to \Mod_{KU}(\Sp)$$
does not preserve all compact maps. {For a recollection of the definition of compact maps and their relevance, we refer to \cref{koprtzehrthrteeh}.}
}

{In the context of localizing motives,  tools and results analogous to the ones in the present note were previously developed by Ramzi--Sosnilo--Winges  \cite{RSW} and Efimov \cite{Efimov_2025}.  We recall that $ \Cat^{\perf} $ is the category of idempotent complete stable categories, and that the category of localizing motives $\mathrm{Mot}_{\mathrm{loc}}$ is the codomain of the initial finitary and bifiber sequence preserving functor with a stable and presentable codomain
$$ U_{\mathrm{loc}} \colon \Cat^{\perf} \to \mathrm{Mot}_{\mathrm{loc}} \ .$$
 The  formal arguments used to show
\cite[Theorem 1.9]{RSW}, stating that the functor $  U_{\mathrm{loc}}  $ preserves countable products, are essentially the same as for  verifying \cref{hrrthregrt} of the present note.} 

 We now describe the contents of this note in order.
 In \cref{khopeertgrtgertge}, we recall various definitions of properties of functors on $C^{*}$-algebras.
 In \cref{khopeertgrtgertge1}, we introduce the notions of $E$- and $KK$-homological functors, the notions of $E$- and $KK$-theory equivalences, and we define the functors in  \eqref{vdfsvsdfvsdfvsdf} by Dwyer-Kan localization; see \cref{hrteoptrhertg2}.
In \cref{jokgkpherthrtgrtgerg}, we recall the constructions of the $E${-} and $KK$-theory functors from 
\cite{bunke:2023aa} which we temporarily call classical; see \cref{hrteoptrhertg1}. In \cref{hrtegertghthrerth}, we
observe that they are examples of $E$- or $KK$-homological functors, respectively. 
In \cref{hkoppertrgertger}, we verify  that
the functors in \eqref{vdfsvsdfvsdfvsdf}  are  $E$- or $KK$-homological functors, respectively; see \cref{erherzhgrt9}. 
In \cref{okpherthertgertg}, we then show, by comparing universal properties, that the functors in  \eqref{vdfsvsdfvsdfvsdf} 
coincide with their classical versions  from \cref{hrteoptrhertg1}. Note that 
this comparison depends on the automatic countable sum theorem \cref{kopherthertgertg}, {which} is
(the only) non-formal result in this note involving the (very) classical analytic models for $E$- or $KK$-theory.
In \cref{kopnerthrtghrgertg} we first show that $\ee$ and $\kk$ do not preserve all
products; see \cref{gwjiogergwerfwerff}. We then introduce the stabilized
product functor \eqref{gewrgwfrewf}, show that it descends to $E$- or $KK$-theory, respectively, and
that the descended functors represent the respective categorical products; see \cref{okphjzhzjrtjzthr}.
The argument employs two different universal properties of the $E$- and $KK$-theory functors. On the one hand, we use that they are Dwyer-Kan localizations at $E${-} or $KK$-theory equivalences admitting calculi of fractions.
On the other hand{,} we use, via  their equivalence with the classical versions, the universal property
\cref{zjoptzjrthzth}.
In Section \cref{zlpherhtrhrgert}, we show that the bootstrap class of $\KK$ is not closed under products; see \cref{hopkethretgertg}{, and  finally in  \cref{koprtzehrthrteeh}, we show \cref{opjwefgioherjgkljsdfg}  asserting that the $K$-theory functor
on $E$-theory does not preserve all compact maps.}

%
%
%

  {\em Acknowledgement: A coversation with Gemini helped to find the present form of the statement in \cref{jzlptzrjtzhztrhr},
  using the maximal tensor product instead of the minimal one. U.B. further thanks Andreas Thom of a discussion about this topic.
 }

\section{Properties of functors on $C^{*}$-algebras}\label{khopeertgrtgertge}

In this section, we recall well-known properties of functors defined on the category of $C^{*}$-algebras.

The category $\nCalg$ of $C^{*}$-algebras is $\aleph_{1}$-presentable.  Its
  full subcategory $\nCalg_{\sepa}$ of separable $C^{*}$-algebras coincides with  its
 subcategory of $\aleph_{1}$-compact objects.  In particular, we have an equivalence
 $$\colim:\Ind_{\aleph_{1}}(\nCalg_{\sepa})\stackrel{\simeq}{\to} \nCalg\ .$$ The   inverse sends $A$ in $\nCalg$ to
 its $\aleph_{1}$-filtered system $(A'\subseteq_{s} A)$ of separable algebras.

The category $\nCalg$ is pointed by the zero algebra $0$.
Let $\cC$ be any $\infty$-category.  
\begin{ddd}
A functor $F:\nCalg\to \cC$ is called reduced if $F(0)$ is a zero object.
\end{ddd}

The category $\nCalg$ is tensored over the category $\Nuc_{\sepa}$ of nuclear separable $C^{*}$-algebras.
We will denote this tensor structure  by $$-\otimes - :\nCalg\times \Nuc_{\sepa} \to \nCalg\ .$$
For every second-countable compact Hausdorff space $W$ we have $C(W)$ in $\Nuc_{\sepa}$ and a map
$\epsilon_{W}:\C\to C(W)$ sending $\lambda$ in $\C$ to the constant function with value $\lambda$.  

\begin{ddd}
A functor $F:\nCalg\to \cC$ is homotopy invariant if{,} for every $A$ in $\nCalg${,} the map $F(A)\simeq F(A\otimes \C)\xrightarrow{F(\id_{A}\otimes \epsilon_{[0,1]})} F(A\otimes C([0,1]))$ is an equivalence.
\end{ddd}

The algebra of compact operators $\bbK$  on the Hilbert space $L^{2}(\nat)$  belongs to $\Nuc_{\sepa}$.
A minimal projection $p$ in $\bbK$ induces, for every  $A$ in $\nCalg$, a left upper coner inclusion
$$A\to A\otimes \bbK\ , \quad a\mapsto a\otimes p\ .$$

\begin{ddd}
A functor $F:\nCalg\to \cC$ is called $\bbK$-stable if it sends left upper corner inclusions to equivalences.
\end{ddd}

  An exact sequence of $C^{*}$-algebras   is a bifibre sequence 
 $$ \xymatrix{A\ar[r]\ar[d] &B \ar[d]^{p} \\0 \ar[r] &C } $$ in $\nCalg$, {i.e.,  {the} square  is simultaneously a pushout and a pullback.
 It is called  semi-split exact, if $p$  admits a completely positive contractive (cpc) linear right-inverse $C\to B$.

 \begin{ddd}
 A functor
 $F:\nCalg\to \cC$ is called exact   (semi-exact) if it is reduced and sends exact   (semi-split exact) sequences
 to bifibre sequences.
 \end{ddd}

 If $(A_{i})_{i\in I}$ is a  family in $\nCalg$, then we define its sum by 
 $$\bigoplus_{I}A_{i}:=\colim_{F\subseteq I} \prod_{i\in F} A_{i}\ ,$$
 where the colimit ranges over the finite subsets $F$ of $I$ and the structure maps are given by the canonical inclusions given by extension of families by  zero.
  
All these notions have versions for separable algebras. In this case,  we only consider   countable sums.

\begin{ddd}
A functor
$F:\nCalg\to \cC$ is called $s$-finitary if it is equivalent to the left-Kan extension of its restriction
to $\nCalg_{\sepa}$ along the inclusion $\nCalg_{\sepa}\to \nCalg$.
\end{ddd}

 {This entails that a functor $F:\nCalg\to \cC$ is $s$-finitary if and only if it preserves all $\aleph_1$-filtered colimits{, or equivalently, if} the canonical map
\[ 
	\colim_{A^\prime \subseteq_s A } F(A^\prime)  \to F(A)
\]
is an equivalence for all $C^\ast$-algebras $A$.}

\section{$E${-} and $KK$-homological functors}\label{khopeertgrtgertge1}

In this section we introduce the notions of homological functors and equivalences for $E${-} and $KK$-theory.

We consider a cocomplete stable $\infty$-category  $\cC$.

\begin{ddd} \label{koptrhertgertget}\begin{enumerate}
\item
 A functor $F:\nCalg\to \cC$ is called $E$-homological if
it is homotopy invariant, $\bbK$-stable, exact and preserves countable sums and $\aleph_{1}$-filtered colimits.
 \item A functor $F:\nCalg\to \cC$ is called $KK$-homological if
it is homotopy invariant, $\bbK$-stable, semi-exact and preserves  countable sums and $\aleph_{1}$-filtered colimits.
 \end{enumerate}
\end{ddd}

\begin{rem}
The only difference between $E${-} and $KK$-theory {lies} in the exactness property. \hB
\end{rem}

\begin{rem}
One can show that an $E$-homological functor preserves all filtered colimits. In fact, exactness and countable sum preservation implies preservation of $\nat$-indexed colimits \cite[Prop. 2.6]{MR2193334}, \cite[Prop. 3.17]{budu}. {Furthermore, it is a general fact that a functor that preserves  $\nat$-indexed and $\aleph_1$-indexed filtered colimits, also preserves all filtered colimits \cite[\href{https://kerodon.net/tag/06EF}{Corollary 06EF}]{kerodon}.}
\hB
\end{rem}

Let $f$ be a morphism in $\nCalg$.
\begin{ddd} \label{hrerhrgterhet}\mbox{}
\begin{enumerate}
\item The map $f$ is an $E$-theory equivalence if it is   inverted by all $E$-homological functors.
\item  The map $f$ is a  $KK$-theory equivalence if it is  inverted by all $KK$-homological functors.
\end{enumerate}
\end{ddd}

\begin{ddd}\label{hrteoptrhertg2} \mbox{}
\begin{enumerate}
\item We define the $E$-theory functor $$\ee:\nCalg\to \EE$$ as the Dwyer-Kan localization at the $E$-theory equivalences. 
\item  We define the $KK$-theory functor $$\kk:\nCalg\to \KK$$ as the Dwyer-Kan localization at the $KK$-theory equivalences. 
\end{enumerate}
\end{ddd}

We record the 
 universal property of these functors:

\begin{kor} \label{zjpoztjzhrtzh}For any $\infty$-category $\cC${,} we have    equivalences
$$\ee^{*}:\Fun(\EE,\cC)\stackrel{\simeq}{\to} \Fun^{W_{E}}(\nCalg,\cC)\ , \quad 
\kk^{*}:\Fun(\KK,\cC)\stackrel{\simeq}{\to} \Fun^{W_{KK}}(\nCalg,\cC)\ .$$
\end{kor} Here  
 the superscript $W_{E}$ and $W_{KK}$ stand for the full subcategories of functors which invert $E$-theory equivalences or $KK$-theory equivalences, respectively.
Further below we will state different universal properties.

\section{The classical $E$ and $\KK$-theory functors}\label{jokgkpherthrtgrtgerg}

In this section, we recall the construction of   $E${-} and $KK$-theory functors as {considered} in \cite{LN}, \cite{bunke:2023aa}. To distinguish them from the functors introduced in \cref{hrteoptrhertg2}, we use the adjective ``classical''. This is justified by the fact that their homotopy categories indeed coincide with the classical $E$- and $KK$-theory functors introduced in \cite{kasparovinvent}, \cite{higson}, \cite{zbMATH04182148}, \cite{Guentner_2000}.

We start with a version of the constructions from \cref{khopeertgrtgertge} for separable $C^{*}$-algebras.
Let $\cC$ be a stable $\infty$-category.
\begin{ddd}   \mbox{}
\begin{enumerate}
\item
 A functor $F:\nCalg_{\sepa}\to \cC$ is called $E_{\sepa}$-homological  if
it is homotopy invariant, $\bbK$-stable and exact.
\item A functor $F:\nCalg_{\sepa}\to \cC$ is called   $KK_{\sepa}$-homological  if
it is homotopy invariant, $\bbK$-stable and  semi-exact.
\end{enumerate}
\end{ddd}
The difference between the $E$- and $KK$-theory cases {again} lies in the exactness conditions.

 Let $f$ be a morphism in $\nCalg_{\sepa}$.
\begin{ddd}  \mbox{}
\begin{enumerate}
\item
The map $f$ is an $E_{\sepa}$-theory equivalence if it is   inverted by all $E_{\sepa}$-homological functors. 
\item The map $f$ is a  $KK_{\sepa}$-theory equivalence if it is   inverted by all   $KK_{\sepa}$-homological functors.  \end{enumerate}
 \end{ddd}

\begin{ddd}\label{hrteoptrhertg1} \mbox{}
\begin{enumerate}
\item We define the $E_{\sepa}$-theory functor $$\ee_{\sepa}:\nCalg_{\sepa}\to \EE_{\sepa}$$ as the Dwyer-Kan localization at the $E_{\sepa}$-theory equivalences. 
\item  We define the $KK_{\sepa}$-theory functor $$\kk_{\sepa}:\nCalg_{\sepa}\to \KK_{\sepa}$$ as the Dwyer-Kan localization at the $KK_{{\sepa}}$-theory equivalences. 
\end{enumerate}
\end{ddd}

The following statements are condensed forms of statements shown in \cite{bunke:2023aa}. 
\begin{theorem} \label{kophertgretgertg9}
The categories $\EE_{\sepa}$ and $\KK_{\sepa}$ are stable and the functors
$\ee_{\sepa}$ and $\kk_{\sepa}$ are $E_{\sepa}$ and $KK_{\sepa}$-homological, respectively.
Furthermore, for every stable $\infty$-category $\cC$ we have equivalences 
\begin{eqnarray*}
\ee^{*}_{\sepa}:\Fun^{\lex}(\EE_{\sepa},\cC)&\stackrel{\simeq}{\to}& \Fun^{E_{\sepa}-\homol}(\nCalg_{\sepa},\cC)\ , \\ \kk^{*}_{\sepa}:\Fun^{\lex}(\KK_{\sepa},\cC)&\stackrel{\simeq}{\to}& \Fun^{KK_{\sepa}-\homol}(\nCalg_{\sepa},\cC)\ .
\end{eqnarray*}
\end{theorem}

\begin{ddd}\label{hrteoptrhertg1}\mbox{}\begin{enumerate} 
\item 
We define the classical $E$-theory functor by
$$\hat \ee:\nCalg \simeq \Ind_{\aleph_{1}}(\nCalg_{\sepa})\xrightarrow{\Ind_{\aleph_{1}}(\ee_{\sepa})} \Ind_{\aleph_{1}}(\EE_{\sepa})=:\hat \EE\ .$$
\item  We define the classical $KK$-theory functor by $$\hat \kk:\nCalg \simeq \Ind_{\aleph_{1}}(\nCalg_{\sepa})\xrightarrow{\Ind_{\aleph_{1}}(\kk_{\sepa})}   \Ind_{\aleph_{1}}(\KK_{\sepa})=:\hat \KK\ .$$
\end{enumerate}
\end{ddd}

The following notions are obtained by dropping the countable sum-preservation in \cref{koptrhertgertget}.
 \begin{ddd}\mbox{}\begin{enumerate}
\item
 A functor $F:\nCalg\to \cC$ is called restricted $E$-homological if
it is homotopy invariant, $\bbK$-stable, exact, and preserves  $\aleph_{1}$-filtered colimits.
 \item A functor $F:\nCalg\to \cC$ is called restricted $KK$-homological if
it is homotopy invariant, $\bbK$-stable, semi-exact,   and preserves $\aleph_{1}$-filtered colimits.
 \end{enumerate}
\end{ddd}

\begin{lem}\mbox{}\begin{enumerate} 
\item 
A functor is restricted $E$ -homological if and only if it preserves $\aleph_{1}$-filtered colimits and its restriction to $\nCalg_{\sepa}$ is
$E_{\sepa}$-homological
\item A functor is restricted $KK$-homological  if and only if it preserves $\aleph_{1}$-filtered colimits and its restriction to $\nCalg_{\sepa}$ is
  $KK_{\sepa}$-homological. \end{enumerate}
\end{lem}

\begin{proof}
We consider the case of $KK$-theory. The case of $E$-theory is analogous.
It is clear that the restriction of a restricted $KK$-homological functor to separable algebras is $KK_{\sepa}$-homological. 
For the converse{,} we use that a  restricted $KK$-homological functor 
is the left-Kan extension of its restriction to separable algebras since it preserves $\aleph_{1}$-filtered colimits.  We then 
  argue, as in \cite{KKG}, \cite{bunke:2023aa}, that  this left Kan-extension inherits the following properties: homotopy invariance, $\bbK$-stability, and semi-exactness.\end{proof}

\begin{kor}\label{zjoptzjrthzth}
For every cocomplete  stable $\infty$-category $\cC$, we have equivalences 
\begin{eqnarray*}
\hat \ee^{*}:\Fun^{\lex, \aleph_{1}\fin}(\hat \EE,\cC)&\stackrel{\simeq}{\to}& \Fun^{\restr-E -\homol}(\nCalg ,\cC)\ , \\  \hat \kk^{*}:\Fun^{\lex,\aleph_{1}\fin}(\hat \KK,\cC)&\stackrel{\simeq}{\to}& \Fun^{\restr-KK -\homol}(\nCalg ,\cC)
\end{eqnarray*}
\end{kor}

Here the superscripts on the right-hand side indicate restricted $E$ or $KK$-homological functors, and $\aleph_{1}\fin$ on the left-hand sides stands for $\aleph_{1}$-filtered colimit preserving.

\begin{theorem}[Automatic countable sums]\label{kopherthertgertg}\mbox{}
\begin{enumerate}
\item The category $\EE_{\sepa}$ admits countable sums and $\ee_{\sepa}$ preserves them.
\item The category $\KK_{\sepa}$ admits countable sums and $\kk_{\sepa}$ preserves them.
\end{enumerate}
\end{theorem}
\begin{proof}

For $E_{\sepa}$ we use the comparison of the homotopy category of $\EE_{\sepa}$ with the group-valued $E$-theory and
\cite[Prop. 7.1]{Guentner_2000} stating the  existence of countable {coproducts}   in the  homotopy category.   
For $\KK_{\sepa}$ this is \cite[Kor. 12.3]{bunke:2023aa}. Alternatively, 
one can use, as above,  the existence of countable coproducts in the classical group-valued $KK$-theory shown in \cite{kasparovinvent}.
 \end{proof}
 
 \begin{kor}\label{hrtegertghthrerth}
 \mbox{}
\begin{enumerate}
\item The functor $\hat \ee$ is $E$-homological. 
\item The functor $\hat \kk$ is $KK$-homological. 
 \end{enumerate}
 In addition, for every cocomplete stable $\infty$-category $\cC$, we have equivalences
 \begin{eqnarray}
\hat \ee^{*}:\Fun^{\colim}(\hat \EE,\cC)&\stackrel{\simeq}{\to}& \Fun^{E-\homol}(\nCalg ,\cC)\ ,   \nonumber\\  \hat \kk^{*}:\Fun^{\colim}(\hat \KK,\cC)&\stackrel{\simeq}{\to}& \Fun^{KK -\homol}(\nCalg ,\cC)\label{jgwegrefwf}
\end{eqnarray}

\end{kor}
\begin{proof} We argue for $KK$-theory. The argument for $E$-theory is the same. 

By \cref{zjoptzjrthzth},
the  functor $\hat \kk$ is restricted $KK$-homological. Since $\kk_{\sepa}$ preserves countable sums by \cref{kopherthertgertg}, the functor
$\hat\kk$ preserves all sums. It  therefore preserves  countable ones. 
The second statement follows from \cref{zjoptzjrthzth} by adding the condition of countable sum preservation on both sides of the equivalence and observing that we get colimit-preserving functors on the left-hand side.
\end{proof}

 \section{Comparison with the classical definitions}\label{hkoppertrgertger}
 
In this section, we show that the functors introduced in \cref{hrteoptrhertg2} coincide with classical
$\infty$-categorical versions of $E$- and $KK$-theory from \cref{hrteoptrhertg1}. 

{Let us call a cartesian square $$\xymatrix{A\ar[r]\ar[d]^{q} & B\ar[d]^{p} \\C \ar[r] & D} $$ in $\nCalg$  $E$-fibrant if $p$   and $q$ are surjections,   and $KK$-fibrant, if they are semi-split surjections. A functor will be called
$h^{E}$- or $h^{KK}$-left exact if it sends $E$-fibrant or $KK$-fibrant cartesian squares to cartesian squares. An exact or semi-exact functor with a stable target is $h^{E}$- or $h^{KK}$-left exact, respectively; see  \cite[Lem. 2.14]{KKG}.}

 \begin{lem}\label{koephethgeteh}
The categories  
$\EE$ and $\KK$ {(from  \cref{hrteoptrhertg2})} are pointed and left-exact. The functor $\ee$ is exact and  $\kk$   is semi-exact. {In addition, we have equivalences
\begin{align}
\ee^{*}&:\Fun^{\lex}(\EE,\cC)\stackrel{\simeq}{\to} \Fun^{W_{E}, h^{E}-\lex}(\nCalg,\cC)\ ,\nonumber\\
\kk^{*}&:\Fun^{\lex}(\KK,\cC)\stackrel{\simeq}{\to} \Fun^{W_{KK},  h^{KK}-\lex}(\nCalg,\cC)\ . \label{kophertgertgetrg}
\end{align} for any pointed left-exact $\infty$-category $\cC$, where $h^{E}-\lex$ or $h^{KK}-\lex$ indicate $h^{E}$- or $h^{KK}$-left exact functors.} 
\end{lem}
\begin{proof} The argument is the same for $\EE$ and $\KK$-theory. We give the argument for $\KK$.
The category $\nCalg$ admits the structure of a category of fibrant objects in the sense of  \cite[Def. 7.4.12 and Def. 7.5.7]{Cisinski:2017},   whose weak equivalences are the $KK$-theory equivalences and whose  fibrations are surjections admitting a cpc split, see    \cite{Uuye:2010aa}.  This implies left-exactness of $\KK$  by
 \cite[Prop. 7.5.6]{Cisinski:2017}.  
 {  By a mapping cylinder construction, any cartesian square in $\nCalg$ is homotopy equivalent to 
  a  $KK$-fibrant cartesian square. 
   A  functor to a pointed left-exact $\infty$-category  that inverts $KK$-theory equivalences (so in particular homotopy equivalences) and  is $h^{KK}$-left exact{,}  sends
 such a square to a  cartesian square in $\cC$.  
 This applies to $\kk$ and
  implies
 the equivalence in \eqref{kophertgertgetrg}.}

 In the case of $E$ theory, for the  fibration{,}  we take all surjections.
\end{proof}

\begin{lem}\label{okphetrtgertg}
The $E$- and $KK$-theory functors (from  \cref{hrteoptrhertg2}) admit symmetric monoidal refinements for the maximal tensor product on $\nCalg$. {The induced tensor structures on $\EE$ and $\KK$ are left-exact in each argument.}
\end{lem}\begin{proof}
For $A$ in $\nCalg$, the endofunctor
$$A\otimes_{\max}-:\nCalg\to \nCalg$$
preserves homotopies, left-upper corner inclusions, exact and  semi-split-exact sequences, sums and filtered colimits. It therefore preserves $E$- and $KK$-theory equivalences. We conclude that the Dwyer-Kan localizations defining $\EE$ and $\KK$ are symmetric monoidal.  {For $D$ in $\nCalg$, the functor 
$D\otimes_{\max} -$ preserves $KK$-fibrant cartesian squares. Consequently, 
$\kk(D\otimes_{\max} -)$ inverts    $KK$-theory equivalences and is $h^{KK}$-left exact .  It therefore descends to a left-exact functor on $\KK$ by  \eqref{kophertgertgetrg}. This implies the second statement for $\KK$.}
\end{proof}

\begin{lem}\label{koopherthertgertg}
The categories $\EE$ and $\KK$ are stable. 
\end{lem}
\begin{proof}

 The proofs are the same for $\EE$ and $\KK$ and similar to the arguments in \cite[Sec. 6 and 7]{bunke:2023aa}. We give the argument for $\KK$.
 The inclusion $0\to C_{0}([0,1))$ is a homotopy equivalence and therefore inverted by all $KK$-homological functors.
 The  semi-split exact sequence $$\xymatrix{C_{0}((0,1)) \ar[r]\ar[d] &C([0,1))  \ar[d]^{\ev_{1}} \\ 0\ar[r] &\C } $$
 is sent by $\kk$  to a fibre sequence which implies that the canonical map $\Omega \kk(\C)\to  \kk(C_{0}((0,1)))$ is an equivalence.  We now consider the semi-split  exact Toeplitz extension
 $$\xymatrix{\bbK\ar[r]\ar[d] &\cT_{0} \ar[d] \\ 0\ar[r] & C_{0}((0,1))} \ .$$ The canonical map $0\to \cT_{0}$ is sent to an equivalence by every $\kk$-homological functor (see \cite[Cor. 6.6]{bunke:2023aa}). It follows that
$\kk(\cT_{0})\simeq 0$ and the canonical map
$\Omega \kk(C_{0}((0,1)))\to \kk(\bbK)$ is an  equivalence. Combining these two equivalences with $\bbK$-stability we {obtain} an equivalence
$$\Omega^{2}\kk(\C)\simeq \kk(\C)\ .$$ {Using that $\otimes_{\max}$ on $\KK$ is left-exact in both arguments,}
this implies that the endofunctor  $\Omega^{2}:\KK\to \KK$ is invertible and therefore the stability of $\KK$.
  \end{proof}
  

\begin{lem}\label{jigowrefwerfgw}
The categories $\EE$ and $\KK$ are cocomplete and $\ee$ and $\kk$ preserve coproducts and send sums to coproducts.
\end{lem}
\begin{proof}
We argue for $KK$-theory. The argument for $E$-theory is the same. 
 
 We first argue that a $KK$-homological functor $F$ preserves coproducts of $KK$-homological equivalences.
 For every   family $(A_{n})_{n\in \nat}$, the canonical map
$$F(\bigsqcup_{\nat} A_{n})\to F(\bigoplus_{\nat} A_n)$$ is an equivalence
{since the canonical map  $\bigsqcup_{\nat} A_{n} \to  \bigoplus_{\nat} A_n$ becomes a homotopy equivalence
after applying $-\otimes_{\max}\bbK$;
see for example \cite[Prop. 7.5]{bunke:2023aa}}. If $(f_{n}:A_{n}\to B_{n})$ is a family of $KK$-homological equivalences, then
$F(\bigoplus_{n\in \nat} f_{n})$ is a  $KK$-homological equivalence since $F$ preserves countable sums.
 It follows that $F(\bigsqcup_{n\in \nat} f_{n})$ is a $KK$-homological equivalence.
 
 We now consider a general index set $I$. If $(f_{i}:A_{i}\to B_{i})_{i\in I}$ is a family of $KK$-homological equivalences, then for every $KK$-homological functor $F${,} we have a commutative square
$$ \xymatrix{\colim_{C\subseteq I} F(\coprod_{C}A_{i})\ar[r]^{\simeq}\ar[d]^{\simeq} &\colim_{C\subseteq I} F(\coprod_{C}B_{i}) \ar[d]^{\simeq} \\  F(\coprod_{i\in I}A_{i})\ar[r] & F(\coprod_{i\in I}B_{i})}\ , $$
where the colimits run over the $\aleph_{1}$-filtered poset of    countable subsets $C$ of $I$.
The upper equivalence follows from the case of countable index sets shown above, and the vertical maps are equivalences  since $F$ preserves $\aleph_{1}$-filtered colimits.  It follows that
the lower horizontal map is an equivalence. This finishes the proof of the fact that $F$ preserves coproducts of $KK$-homological equivalences.

 Since the Dwyer-Kan localization $\nCalg\to \KK$  at the $KK$-theory equivalences is controlled by a structure of a category of fibrant objects (see the proof of \cref{koephethgeteh}), by \cite[7.7.1]{Cisinski:2017}
 the left vertical map in the commutative {square} $$\xymatrix{\prod_{i\in I}\nCalg\ar[r]^{\sqcup}\ar[d]^{\prod_{i\in I}\kk} & \nCalg\ar[d]^{\kk} \\  \prod_{I} \KK\ar@{-->}[r]^{\sqcup} & \KK } $$   {is} the Dwyer-Kan localization at the families of $KK$-theory equivalences.
   Since the right-down composition inverts families of $KK$-theory equivalences, we obtain the dashed factorization. Again, since the vertical maps are localizations, the $(\coprod,\diag)$-adjunction on the level of $C^{*}$-algebras 
 descends to an adjunction on the level of $KK$-categories.  Therefore the lower horizontal  functor  represents the coproduct.  We conclude that $\KK$ admits  coproducts and that $\kk$ preserves them. 
 
 The argument shows that  $\kk$   preserves all sums. 
 \end{proof}

An $E$- or $KK$-homological functor $F$ to a cocomplete stable $\infty$-category $\cC$  inverts $E$-homological equivalences or $KK$-homological equivalences by 
definition. By  \cref{zjpoztjzhrtzh}, we  obtain  unique factorizations
$$\xymatrix{\nCalg\ar[rr]^{F}\ar[dr]_{\ee}&&\cC\\&\EE\ar@{-->}[ur]_{\hat F}&}\ , \quad \xymatrix{\nCalg\ar[rr]^{F}\ar[dr]_{\ee}&&\cC\\&\KK\ar@{-->}[ur]_{\hat F}&}\ ,$$
respectively.

\begin{lem}\label{kohetrrtgertg}
In both cases{,} the functor $\hat F$ preserves colimits. 
\end{lem}
\begin{proof}
We consider the case of $KK$-theory. The case of $E$-theory is analogous. 
Since $\kk$ and $F$ preserves sums, the functor $\hat F$ preserves sums.
{Furthermore, since $\KK$ is stable and $F$ is semi-exact it belongs to the codomain of the equivalence in \eqref{kophertgertgetrg}. Therefore, the functor $\hat F$ is left exact.}
By stability, it preserves all colimits. 
 \end{proof}
 
 \begin{lem}\label{zkopjtzjrtzhrtzhr}\mbox{}
In both cases{,} the collection of  functors $\hat F$ for all $E$- or $KK$-homological functors $F$, respectively, is jointly
conservative. 
 \end{lem}
 \begin{proof}
 We consider the case of $KK$-theory. The case of $E$-theory is analogous. Let $\hat f$ be a morphism in $\KK$.
Since the localization   $\kk$ is controlled by a category of fibrant objects (see the proof of \cref{koephethgeteh}),
there exists a morphism $f$ in $\nCalg$ such that $\kk(f)\simeq \hat f$.
Assume that $\hat F(\hat f)$ is an equivalence for all $KK$-homological functors $F$. Then
$\hat F(\hat f)\simeq F(f)$ is an equivalence for all such functors. Therefore $f$ is a $KK$-homological equivalence and   $  \hat f$ is an equivalence.
 \end{proof}
 
 We now finally show:
 \begin{lem}\label{erherzhgrt9}
 \mbox{}
 \begin{enumerate}
 \item The functor $\ee:\nCalg\to \EE$ is an $E$-homological functor.
 \item The functor $\kk:\nCalg \to \KK$ is a $KK$-homological functor.
  \end{enumerate}
\end{lem}
\begin{proof}
We discuss the case of $KK$-theory.  The case of $E$-theory is analogous. 
The category $\KK$ is stable (\cref{koopherthertgertg}) and cocomplete (\cref{jigowrefwerfgw}). 
We already know that $\kk$ is homotopy invariant and $\bbK$-stable. 
By \cref{koephethgeteh}, it is exact.  By \cref{jigowrefwerfgw} it preserves all sums, so in particular countable ones.
It remains to show that $\kk$ preserves $\aleph_{1}$-filtered colimits.
Let $(A_{i})_{i\in I}$ be an  $\aleph_{1}$-filtered  diagram in $\nCalg$ and consider the comparison map
\begin{equation}\label{regwerfwerf}\colim_{i\in I} \kk(A_{i})\to \kk(\colim_{i\in I} A_{i})\ .
\end{equation}
Let $F$ be any $KK$-homological functor. We apply $\hat F$ and obtain  
$$\hat F(\colim_{i\in I} \kk(A_{i}))\simeq \colim_{i\in I} F(\kk(A_{i}))\simeq   \colim_{i\in I}  F(A_{i})\simeq F(\colim_{i\in I} A_{i})\simeq \hat F(\kk(\colim_{i\in I} A_{i}))\ ,$$ where we use \cref{kohetrrtgertg} for the first equivalence 
and the fact that $F$ preserves $\aleph_{1}$-filtered colimits for the third equivalence. 
By \cref{zkopjtzjrtzhrtzhr}, we conclude that the comparison map \eqref{regwerfwerf} itself is an equivalence.
 \end{proof}

{\begin{kor} \label{iogjoegwerfwerfwerf}.
We have equivalences 
 \begin{align*}
 \ee^{*}&:\Fun^{\colim}(\EE,\cC)\stackrel{\simeq}{\to} \Fun^{E-\homol}(\nCalg,\cC)\ ,\\
  \kk^{*}&:\Fun^{\colim}(\KK,\cC)\stackrel{\simeq}{\to} \Fun^{KK-\homol}(\nCalg,\cC)\end{align*}
  for any cocmplete stable $\infty$-category $\cC$.
\end{kor}
\begin{proof}
We discuss the $KK$-theory case. The $E$-theory case is analogous.
It follows from \cref{erherzhgrt9} that the restriction along $\kk$ takes values in the subcategory as asserted. 
Since   $\kk$ is a localization, the restriction functor $\kk^{*}$  is fully faithful. It remains to show that it is  essentially surjective.
To this end, we start from  \cref{koephethgeteh}. Since we assume that $\cC$ is stable, as noted before \cref{koephethgeteh} the condition of  $h^{KK}$-left exactness reduces to  semi-exactness,  which are part of the definition of an  $KK$-homological functor.  Since these functors invert all  $KK$-theory equivalences, using  \cref{koephethgeteh} they are restrictions along  $\kk$  of left-exact functors.  By \cref{kohetrrtgertg} the  preimages actually preserve all colimits.
\end{proof}}

\begin{kor}\label{jerhrthgergert}\mbox{}
\begin{enumerate}
\item The functor $\ee$ inverts precisely the $E$-theory equivalences.
\item The functor $\kk$ inverts precisely the $KK$-theory equivalences.
\end{enumerate}
\end{kor}
\begin{proof}
We discuss the case of $KK$-theory.  The case of $E$-theory is analogous. 
Since $\kk$ is $KK$-homological by \cref{erherzhgrt9} it inverts the $KK$-theory equivalences. 
Conversely, assume that a morphism $f$ in $\nCalg$ is inverted by $\kk$.
Since every $KK$-homological functor factorizes over $\kk$, the morphism $f$ is inverted by every $KK$-homological functor  {and is therefore   a $KK$-theory} equivalence.
\end{proof}

 \begin{kor}\label{okpherthertgertg}\mbox{}
 \begin{enumerate}
 \item The functors $\ee:\nCalg\to \EE$ and $\hat \ee:\nCalg\to \hat \EE$ are canonically equivalent.
 \item The functors $\kk:\nCalg\to \KK$ and $\hat \kk:\nCalg\to \hat \KK$ are canonically equivalent.
 \end{enumerate}
 \end{kor}
\begin{proof}
{The functors have same universal property: Compare \cref{iogjoegwerfwerfwerf} with \cref{hrtegertghthrerth}. }
%
%
%
\end{proof}

\begin{rem}\label{jtzjtzhtrzhrjztjr}
	As a consequence of \cref{okpherthertgertg}, {the functors $\ee$ and $\kk$} also satisfy the universal properties described in \cref{zjoptzjrthzth} {(in place of $\hat \ee$ or $\hat \kk$, respectively)}. This is the most important non-formal ingredient for the proof of \cref{okphjzhzjrtjzthr} {below}.
	
	 In comparison, in the proof  in \cite{RSW} of the fact that $U_{\mathrm{loc}}: \Cat^{\perf} \to \mathrm{Mot}_{\mathrm{loc}}$ preserves countable products, the most crucial ingredient is that
	 this functor
	has the additional universal property of being the initial $\aleph_1$-finitary ({in contrast to being finitary which is the defining property}) and bifiber sequence preserving functor with a codomain that is stable and admits $\aleph_1$-filtered colimits \cite[Theorem 1.8]{RSW}. The latter result relies on the computation of mapping spectra in localizing motives due to Efimov \cite[Theorem 4.3]{Efimov_2025}.
	\hB
\end{rem}

\begin{rem}\label{jtzrhzrthrh}{
In this remark we explain the precise  relation between the functors from  \cref{okhperthtrege} with stable $\infty$-category valued $KK$- or $E$-theory functors from previous literature.}

{The $KK$-theory functor in \cite{LN} is defined on separable $C^{*}$-algebras and coincides with the restriction of the $KK$-theory  functor from  \cref{okhperthtrege} to separable $C^{*}$-algebras.} 

{The $KK$-theory functor in  \cite{KKG}  and the $KK$- and $E$-theory functors in  \cite{bunke:2023aa}
are defined on all $C^{*}$-algebras, but do not preserve countable sums. 
They admit natural transformations to the functors from   \cref{okhperthtrege}  which are in fact
left Bousfield localizations forcing the preservation of countable sums.}

{The paper \cite{budu} mainly deals with $E$-theory for separable algebras, but it mentions the extension
to all $C^{*}$-algebras that coincides with the functor from   \cref{okhperthtrege}. }

{Upon restriction to separable $C^{*}$-algebras, the $KK$-theory functor constructed in  \cite{Datta:2025aa} also coincides 
with the $KK$-theory functor from   \cref{okhperthtrege}.} \hB
\end{rem}

  \section{Products in $E$-theory and $KK$-theory}\label{kopnerthrtghrgertg}
  
   We start with demonstrating by an example that the  $E${-} and $KK$-theory functors do not preserve all products.
   We then show that a stabilized countable product functor descends to $E${-} or $KK$-theory and represents the cartesian product there.
  
%
 
  \begin{ex}[$\ee$ and $\kk$ do not preserve all products]\label{gwjiogergwerfwerff}
  
  We give the argument for $E$-theory. The case of $KK$-theory is analogous.
  We employ the fact  that the topological $K$-theory functor is represented in $E$-theory by the tensor unit $\ee(\C)$. More concretely, 
   for every $A$ in $\nCalg$, we have an equivalence \begin{equation}\label{werfewrgwrg}K(A)\simeq \EE(\C,A)\ ,
\end{equation}
  where $K(A)$ denotes the topological $K$-theory spectrum of $A$, and we use the abbreviation  $$\EE(-,-):=\map_{\EE}(\ee(-),\ee(-))\ .$$
  A detailed argument for this fact is given  \cite[Sec. 9]{bunke:2023aa}.
 If the functor $\ee$ would preserve products, then the $K$-theory functor
 $K:\nCalg\to \Sp$  also  would preserve products. In fact, for any family $(A_{i})_{i\in I}$ in $\prod_{I}\nCalg$, we would have the marked equivalence in 
 $$K(\prod_{i\in I}A_{i})\stackrel{\eqref{werfewrgwrg}}{\simeq} \EE(\C,\prod_{i\in I} A_{i}) \stackrel{!}{\simeq}  \map_{\EE}(\ee(\C),\prod_{i\in I} \ee(A_{i}))\simeq  \prod_{i\in I} \EE(\C,A_{i})\stackrel{\eqref{werfewrgwrg}}{\simeq} \prod_{i\in I} K(A_{i})\ .$$
 However,  it is well-known that $K$ does not preserve products. Consider, for example, 
   the constant family $(\C)_{\nat}$ in $\prod_{\nat}\nCalg$.
 We then have $\pi_{0}( \prod_{\nat} K(\C))\cong \prod_{\nat}\Z$.  In contrast, 
 $\pi_{0}(K(\prod_{\nat} \C))$ is   the proper subgroup of this product consisting of the bounded sequences. \hB
  \end{ex}
  
{On the other hand, we have:
 \begin{theorem}[\cite{Willett_2012}]  \label{giwoerpgwerferfw}For every set $I$ we have a commutative square
 $$\xymatrix{\prod_{I}\nCalg\ar[rr]^{\prod_{I}K}\ar[d]^-{\prod_{I}(-\otimes \bbK) } &&\prod_{I}\Sp \ar[d]^{\prod_{I}} \\  \nCalg\ar[rr]^{K} &&\Sp }\ . $$
 \end{theorem}
 See also \cite[Prop. 8.108]{buen} for a complete proof of this result.}
This indicates
that the  reason for the failure of the functors $\ee$ or $\kk$ to preserve products may be the fact  that ${\bbK}$-stabilization  is not compatible with products. In order to circumvent this problem,
 for every set $I$, we define the stable product functor \begin{equation}\label{gewrgwfrewf}\prod^{s}_{I}:\prod_{I}\nCalg\to \nCalg\ , \quad (A_{i})_{i\in I} \mapsto \prod_{i\in I} (A_{i}\otimes \bbK)\ .
\end{equation}


\begin{prop}\label{hertgergetrbert}  
If the set $I$ is  countable, 
 then the  stable product  functor \eqref{gewrgwfrewf}
 sends families of $E$-theory  ($KK$-theory) equivalences to $E$-theory ($KK$-theory) equivalences.
  \end{prop} 
\begin{proof}

 \begin{lem} \label{kopherthertgergertg} If $I$ is a countable set, then the functor 
$$ \prod_{i\in I}(-\otimes {\bbK}):\nCalg\to \nCalg$$
 preserves   $E$-theory  and $KK$-theory   equivalences. \end{lem}\begin{proof} We discuss the case of $E$-theory.  The case of $KK$-theory is analogous. 
 The functor $\prod_{i\in I}(-\otimes {\bbK})$ preserves homotopy equivalences. Here it is important that we take the same homotopy in every factor.  
 If $f$ is a left upper corner inclusion, then 
 $f\otimes \id_{\bbK}$ is a  homotopy equivalence.
 It follows that the functor  $\prod_{i\in I}(-\otimes{\bbK}) $ sends left upper corner inclusions to homotopy equivalences.
 The functor furthermore preserves exact sequences, since $-\otimes \bbK$ preserves exact sequences, and since  a product of exact sequences is again an exact sequence. 
 Finally, since $I$ is countable and $-\otimes \bbK$ preserves filtered colimits, the functor in question preserves
 $\aleph_{1}$-filtered colimits. 

 If $F$ is an $E$-homological functor, then the composition
$$F\circ \prod_{i\in I}(-\otimes {\bbK}):\nCalg\to \EE$$ is a restricted $E$-homological functor. 
By \cref{zjoptzjrthzth} {(see also  the first part of \cref{jtzjtzhtrzhrjztjr})}, it factorizes over $\ee$ and therefore inverts $E$-theory equivalences. 
Using \cref{zkopjtzjrtzhrtzhr}, we conclude that $\prod_{i\in I}(-\otimes {\bbK})$ preserves $E$-theory equivalences. 

In the case of $KK$-theory{,} we use that a product of semi-split exact sequences is again semi-split exact. 
\end{proof}

We now complete the proof of the proposition.  Let $(f_{i}:A_{i}\to B_{i})_{i\in I}$ be a countable family of $E$-theory equivalences in $\nCalg$.
%
 Then we consider the commutative diagram \begin{equation}\label{grtgertgertgertg}\xymatrix{\prod_{j\in I}  (A_{j}\otimes \bbK)\ar[rrr]^{\prod_{j}f_{j}\otimes \id_{\bbK}}\ar[d]^{(1)}&&&\prod_{j\in I}  (B_{j}\otimes \bbK)\ar[d]^{(2)}\\
\prod_{ j\in I} \bigoplus_{i\in I} (A_{i}\otimes \bbK)\ar[rrr]^{\prod_{j}\bigoplus_{i}f_{i}\otimes \id_{\bbK}}\ar[d]^{(3)}&&&  \prod_{j\in I} \bigoplus_{i\in I} (B_{i}\otimes \bbK) \ar[d]^{(4)} \\\prod_{k\in I}  (A_{k}\otimes \bbK)\ar[rrr]^{\prod_{k}f_{k}\otimes \id_{\bbK}}&&& \prod_{k\in I}  (B_{k}\otimes \bbK)}\ .
\end{equation} The   map (1) is the product of  the family of inclusions 
$$((A_{j}\otimes \bbK)\to \bigoplus_{i\in I} (A_{i}\otimes \bbK))_{j\in I}\ .$$ The map (2) is defined analogously.
The  map (3) is induced by the family of  maps
$$( \prod_{j\in J}\bigoplus_{i\in I} (A_{i} \otimes \bbK)\to A_{k}\otimes \bbK)_{k\in I}\ ,$$  
whose member with index $k$ is the projection onto the component with $(j,i)=(k,k)$. Again, the map $(4)$ is defined analogously.
The vertical compositions in \eqref{grtgertgertgertg} are  identities. 

Since a sum of $E$-theory equivalences is again an $E$-theory equivalence and we can interchange $-\otimes \bbK$ and the sum, it follows from \cref{kopherthertgergertg} that the middle horizontal map in \eqref{grtgertgertgertg} is an equivalence.
Since a retract of an $E$-theory equivalence is again an $E$-theory equivalence,
we conclude that $\prod^{s}_{i\in I} f_{i}$ is an  $E$-theory equivalence. \end{proof}

\begin{theorem}\label{okphjzhzjrtjzthr}
We assume that $I$ is  a countable set.
\begin{enumerate}
\item The stable product functor \eqref{gewrgwfrewf} descends to $E$-theory, and the descended functor represents the categorical product in $\EE$. 
\item The stable product functor \eqref{gewrgwfrewf} descends to $KK$-theory, and the descended functor represents the categorical product in $\KK$.
\end{enumerate}
\end{theorem}
\begin{proof}
We consider the case of $E$-theory. The case of $KK$-theory is analogous.
We first consider the square
\begin{equation}\label{hrtgertgtebherth}\xymatrix{\prod_{I} \nCalg \ar[d]^{\prod_{I}\ee}\ar[r]^{\prod_{I}^{s}}&  \nCalg\ar[d]^{\ee}\\ \prod_{I}\EE \ar@{-->}[r]^{\prod_{I}^{s}}&\EE
}\ .
\end{equation} 
Since the Dwyer-Kan localization $\ee$ is controlled by the structure of a category of fibrant objects (see the proof of \cref{koephethgeteh}),  by \cite[7.7.1]{Cisinski:2017}
 the left vertical map   is the Dwyer-Kan localization at the families of $E$-theory equivalences.
Since the right-down composition inverts families of $E$-theory equivalences by \cref{hertgergetrbert}, we obtain the dashed factorization.

It remains to show that it represents the categorical product. 
{By definition of the categorical product, this is equivalent to showing the existence of an adjunction}
\begin{equation}\label{greogkpwerfwerfw}\underline{(-)}:\EE\rightleftarrows \prod_{I} \EE:  \prod^{s}_{I}\ .
\end{equation} 
We start with defining the   unit and the counit.  For the   unit, we 
  consider the natural transformation
 $$ \id\to \prod_{I}^{s}\circ (\underline{-}  ):\nCalg \to \nCalg\ ,$$ that on
 $A$ in $\nCalg $   is given by the map
 $$ A\xrightarrow{\diag}\prod_{I} \underline{A} \xrightarrow{\prod (-\otimes p)}\prod_{I} \underline{A\otimes \bbK}\cong \prod_{I}^{s}\underline{A}\ ,$$
 where $p$ is a rank-one projection in $\bbK$.  By  \cref{hertgergetrbert},
 it descends to a transformation
   $$\alpha: \id\to      \prod^{s}_{I}\circ (\underline{- }):\EE \to \EE \ .$$

 For the counit, we consider the natural transformation
 $$   \underline{\prod_{I}-\otimes \bbK } \to \prod_{I} (-\otimes \bbK) :\prod_{I} \nCalg \to \prod_{I}\nCalg \ ,$$ that
 on a family $(A_{i})_{i}$ in $\prod_{I} \nCalg $ it is given by
 $$\ (p_{i})_{i}: \underline{\prod_{i\in I}A_{i}\otimes \bbK }\to ( A_{i}\otimes \bbK)_{i}\ .$$
 By \cref{hertgergetrbert}, it  descends to 
 $$ \omega:\underline{\prod^{s}_{i\in I} -  }\to   \id :\prod_{I}\EE \to \prod_{I} \EE\ .$$
 We now check the triangle identities:
 We consider the composition
 $$\prod_{I}^{s}\xrightarrow{\alpha\circ \prod_{I}^{s}}  \prod_{I}^{s}\circ (\underline{-}) \circ   \prod_{I}^{s}\xrightarrow{  \prod_{I}^{s}\circ \omega}  \prod_{I}^{s}:  \prod_{I}\EE\to  \EE\ .$$
 This  map is induced by the map
 $$ \prod_{i\in I} (A_{i}\otimes \bbK)\xrightarrow{\diag\otimes p} \prod_{j\in I} (  \underline{\prod_{i\in I} (A_{i}\otimes \bbK))\otimes \bbK}
 \xrightarrow{\prod_{j\in I}p_{j}}   \prod_{j\in I} (A_{j}\otimes \bbK\otimes \bbK)$$
 on the level of $C^{*}$-algebras, 
 This is the stable product of the family of  left upper corner inclusions  $(-\otimes p:A_{i}\to A_{i}\otimes \bbK)_{i\in I}$. Since the stable product sends families of  $E$-theory equivalences to equivalences by \cref{hertgergetrbert}, it is an equivalence.

 We now calculate the composition
 $$(\underline{-})\xrightarrow{(\underline{-})\circ \alpha} (\underline{-})\circ    \prod_{I}^{s}\circ (\underline{-})\xrightarrow{\omega\circ (\underline{-})} (\underline{-})  :\EE\to \prod_{I}\EE\ .$$
  This map is induced by the map 
 $$\underline{A}\xrightarrow{\underline{\diag\otimes p}} \underline{\prod_{I}A\otimes \bbK} \xrightarrow{(p_{i})_{i}}  \underline{A\otimes \bbK}$$ on the level of $C^{*}$-algebras.
 This map is componentwise the  left upper corner inclusion and  sent by $\ee$ to the constant family of $\id_{\EE}$.
 
 Since one of the triangle identities is satisfied and the other one is an equivalence, the latter is actually also the identity.
 This {completes}  the construction of the adjunction \eqref{greogkpwerfwerfw}, {and therefore  the proof of \cref{okphjzhzjrtjzthr}.}
%
 \end{proof}

{The following is the $E$- and $KK$-theoretic analog of \cref{giwoerpgwerferfw}:}
\begin{kor} \label{okphrethertgertge1}If $I$ is countable, then we have  commutative squares
 $$\xymatrix{\prod_{I}\nCalg\ar[rr]^{\prod_{I}\ee}\ar[d]^-{\prod_{I}^{s}} &&\prod_{I}\EE \ar[d]^{\prod_{I}} \\  \nCalg\ar[rr]^{\ee} &&\EE }\ , \quad \xymatrix{\prod_{I}\nCalg\ar[rr]^{\prod_{I}\kk}\ar[d]^-{\prod_{I}^{s}} &&\prod_{I}\KK \ar[d]^{\prod_{I}} \\  \nCalg\ar[rr]^{\kk} &&\KK }\ $$
 \end{kor}

 \section{Products and the bootstrap class}\label{zlpherhtrhrgert}
 
 In this section{,} we introduce the bootstrap classes in $\EE$ and $\KK$ and discuss the question whether they are closed under taking infinite products. 
 In \cref{okphetrtgertg}{,} we have seen that the maximal tensor product $\otimes_{\max}$ on $\nCalg$
 induces symmetric monoidal structures on $\EE$ and $\KK$ such that there are 
 symmetric monoidal refinements of the functors
 $$\ee:\nCalg\to \EE\ , \quad \kk:\nCalg\to \KK\ .$$
As a consequence of \cref{kohetrrtgertg}, the symmetric monoidal structure $\otimes_{\max}$ on $\EE$ and $\KK$ is cocontinuous in each argument. The tensor units are given by 
$\ee(\C)$ and $\kk(\C)$, respectively.  Using, \eqref{werfewrgwrg} {(or $K(A)\simeq \KK(\C,A)$ in the case of $\
KK$-theory)},  the classical fact that the $K$-theory functor for
$C^{*}$-algebras preserves coproducts, and that the functors $\ee$ and $\kk$ preserve coproducts by \cref{jigowrefwerfgw}, one sees 
 that $\ee(\C)$ and $\kk(\C)$ are compact objects in the respective categories. We have the commutative ring spectrum
$$KU:=K(\C)\simeq  {\EE(\C,\C)\simeq \KK(\C,\C)} .$$
The existence of the symmetric monoidal right Bousfield localizations \eqref{gwerfwerrewgwergerwgwe} is now a formal consequence of these facts. \begin{ddd} The bootstrap classes in $\EE$ and $\KK$ are defined as the 
 essential images of the respective left-adjoints $b^{E}$ and $b^{KK}$.\end{ddd} They are copies of  the category ${\Mod_{KU}(\Sp)}$ in $\EE$ and $\KK$. 
 
 The functors $K^{E}$ and $K^{KK}$ are lax symmetric monoidal. \begin{ddd} The (max-)Künneth classes in $\EE$ and $\KK$  are defined as the full subcategories on the following sets of objects:
 $$\cK_{\max}^{E}:=\{A\in \EE\mid \forall B\in \EE :  K^{E}(A)\otimes_{KU}K^{E}(B)\stackrel{\simeq}{\to} K^{E}(A\otimes_{\max}B)\}\ , $$
 $$\cK_{\max}^{KK}:=\{A\in \KK\mid \forall B\in \KK :  K^{KK}(A)\otimes_{KU}K^{KK}(B)\stackrel{\simeq}{\to}  K^{KK}(A\otimes_{\max}B)\}\ . $$\end{ddd}
 Since the tensor product is cocontinuous {in each variable,} the Künneth classes are localizing subcategories. Furthermore,
since the tensor unit belongs to the Künneth classes,
the bootstrap classes are contained in the Künneth classes. 

%
%
 
\begin{theorem}\label{okhpertertger}
We have $\prod_{\nat }\ee(\C)\not\in \cK_{\max}^{E}$ and $\prod_{\nat }\kk(\C)\not\in \cK_{\max}^{KK}$.
\end{theorem}

Before proving this result we derive the desired consequences.

\begin{kor}\label{hopkethretgertg}
\mbox{}
\begin{enumerate}
\item\label{hopkethretgertg1} The $E$- and $KK$-bootstrap classes are not closed under countable  products.
\item The functors $b^{E}$ and $b^{KK}$ do not preserve countable products.
\end{enumerate}
\end{kor}
\begin{proof} We argue for the $KK$-case. The $E$-case is analogous.
The first assertion follows since the $KK$-bootstrap class is contained in  $\cK_{\max}^{KK}$ and $\kk(\C)$ belongs to the $KK$-bootstrap class. 
For the second assertion we use that $b^{KK}(KU)\simeq \kk(\C)$. If $b^{KK}$ would preserve infinite products, then
$\prod_{\nat }\kk(\C)\simeq b^{KK}(\prod_{\nat}KU)$ would belong to the bootstrap class.
\end{proof}

\begin{proof}[Proof of  \cref{okhpertertger}]  We provide the argument for the $KK$-case. The $E$-case is analogous.
 We show  that the comparison map
 \begin{equation}\label{fweqrfewrfwergr}K^{KK}(\prod_{\nat} \kk(\C)) \otimes_{KU} K^{KK}(\prod_{\nat} \kk(\C)) \to K^{KK}(\prod_{\nat} \kk(\C)\otimes_{\max}\kk(\C))
\end{equation} is not an equivalence.
 To this end we construct an element in $\pi_{0}K^{KK}(\prod_{\nat} \kk(\C)\otimes_{\max}\kk(\C))$ which is not in the image.
 
 Since $K^{KK}$ is a right-adjoint, it preserves products. Using $K^{KK}(\kk(\C))\simeq KU$ the domain of the comparison map  \eqref{fweqrfewrfwergr}  is equivalent to
 $$\prod_{\nat} KU\otimes_{KU}\prod_{\nat}KU\ .$$
 We now use the Künneth formula for $KU$-modules.  For $M,N$ in $\Mod_{KU}(\Sp)$
 {this is}  an exact sequence of abelian groups
 {$$0\to  \pi_{0}M\otimes \pi_{0}N \oplus \pi_{1}M\otimes \pi_{1}N \to \pi_{0}(M\otimes_{KU}N) \to  \pi_{0}M*\pi_{1}N\oplus \pi_{1}M*\pi_{0}N \to 0\ ,$$}
 where {$- * -$ denotes the abelian groups $\Tor(-,-)$}. Using further that
 $$\pi_{*}\prod_{\nat}KU\cong \left\{\begin{array}{cc} \prod_{\nat} \Z&*=0\\0 &*=1  \end{array} \right.\ ,
 $$
 we obtain for the domain of the map in \eqref{fweqrfewrfwergr}: \begin{equation}\label{gwerferfrgw}\pi_{0}(K^{KK}(\prod_{\nat} \kk(\C)) \otimes_{KU} K^{KK}(\prod_{\nat} \kk(\C)))\cong \prod_{\nat}\Z\otimes  \prod_{\nat}\Z\ .
\end{equation}
 We now analyse the target. Using \cref{okphjzhzjrtjzthr}{,} we have
 $$\prod_{\nat} \kk(\C)\simeq \kk(\prod_{\nat} \bbK)\ .$$ 
 Since $\kk$ is symmetric monoidal, we then obtain
 $$\prod_{\nat} \kk(\C)\otimes_{\max} \prod_{\nat} \kk(\C)\simeq 
\kk( \prod_{\nat}  \bbK)\otimes_{\max} \kk(\prod_{\nat} \bbK)\simeq \kk(\prod_{\nat} \bbK\otimes_{\max} \prod_{\nat} \bbK)\ .$$
 The $KU$-module in the target  of the comparison map  \eqref{fweqrfewrfwergr}  is thus given by
$$ K(\prod_{\nat} \bbK\otimes_{\max} \prod_{\nat} \bbK)\ .$$
For $n$ in $\nat$ let $p_{n}:\prod_{\nat}\cdots\to \cdots$ denote the projection to the $n$th-factor.
For $(m,n)$ in $\nat\times \nat$ we have a
  commutative  diagram
\begin{equation}\label{zjoptztzhrtzhr}\xymatrix{\ar@/_5cm/[ddd]^{\eqref{gwerferfrgw}}_{\cong}\pi_{0}(K^{KK}(\prod_{\nat} \kk(\C)) \otimes_{KU} K^{KK}(\prod_{\nat} \kk(\C)))\ar[r]^-{\pi_{0}(\eqref{fweqrfewrfwergr})}\ar[dd]_{\cong}^{\pi_{0}(KK^{K}(p_{m})\otimes_{KU}K^{KK}(p_{n}))}&\ar[d]^{!}_{\cong } \pi_{0}(K^{KK}(\prod_{\nat} \kk(\C)\otimes_{\max}\kk(\C)))\\&\pi_{0}K(\prod_{\nat}\bbK\otimes_{\max}\prod_{\nat}\bbK)\ar@/^-2cm/@{-->}[dd]\ar[d]^{\pi_{0}K(p_{m}\otimes_{\max} p_{n})}\\ \pi_{0}( KK^{KK}(\kk(C))\otimes_{KU} K^{KK}(\kk(\C)))\ar[dr]^{\cong} &\ar[d]^{\cong} \pi_{0} K(\bbK\otimes_{\max} \bbK )\ar[d] \\\prod_{\nat}\Z\otimes \prod_{\nat}\Z\ar[r]_{p_{m}\otimes p_{n}}&\Z}\ .\end{equation}
The marked isomorphism was justified above.
The collection of projection maps (indicated by the dashed arrow above) for all $(m,n)$ in $\nat\times \nat$ induces a homomorphism
\begin{equation}\label{gewrfwerfwergwerg}\pi_{0}K(\prod_{\nat} \bbK \otimes_{\max}\prod_{\nat} \bbK )\to \prod_{\nat\times \nat} \Z\ , \quad q\mapsto (q_{m,n})_{(m,n)\in \nat\times \nat}\ .
\end{equation} 
\begin{prop}\label{jzlptzrjtzhztrhr}
There exists a class   $q$ in $\pi_{0}K(\prod_{\nat} \bbK\otimes_{\max} \prod_{\nat} \bbK)$
whose image  under \eqref{gewrfwerfwergwerg}   satisfies 
  $$q_{m,n}=\left\{\begin{array}{cc} 1 & m=n\\ 0  &  m\not=n \end{array} \right.\ .
 $$
\end{prop}
We first complete the proof of \cref{okhpertertger} assuming this proposition.

 The $K$-theory class $q$ is not in the image of the comparison map \eqref{fweqrfewrfwergr}.
 In fact, by the commutativity of \eqref{zjoptztzhrtzhr},  the image of  \eqref{fweqrfewrfwergr} is sent by the composition of the marked map in  \eqref{zjoptztzhrtzhr} with \eqref{gewrfwerfwergwerg}  to finite-rank integral $\nat\times \nat$-matrices, while
 $q$  is sent to  the identity matrix which does not have finite rank.
\end{proof}

\begin{proof}[Proof of \cref{jzlptzrjtzhztrhr}]

 The following construction is essentially due to \cite{Ozawa_2003}.
 We can choose a finitely generated group  $G$   with property $T$ and a family
 of pairwise non-isomorphic finite-dimensional unitary irreducible representations $(V_{i},\rho_{i})_{i\in \nat}$.
 We then consider the Hilbert space   $ \bigoplus_{i\in \nat } V_{i}$ with the unitary representation that sends
 $g$ to   $\pi(g):= \prod_{i\in \nat} \rho_{i}(g)$ in $\prod_{i\in \nat } \bbK(V_{i})$.
  By $(\bar V,\bar \pi(g))$ we denote the dual representation. 
  For every group element $g$ in $G${,} we form the unitary
   $ \sigma(g):= \bar \pi(g)\otimes_{\max}  \pi(g)$  in 
 $ \prod_{m\in \nat} \bbK(\bar V_{m})\otimes_{\max}  \prod_{n\in \nat } \bbK(V_{n})$. 
By the universal property of the maximal group $C^{*}$-algebra{,}  this  homomorphism   extends to a homomorphism of $C^{*}$-algebras
 $$\sigma:C^{*}_{\max}(G)\to   \prod_{m\in \nat} \bbK(\bar V_{m})\otimes_{\max}  \prod_{n\in \nat } \bbK(V_{n})\ .$$
 For every non-empty finite symmetric subset $S$ of $G${,} we consider the selfadjoint operator
 $$B_{S}:=\frac{1}{|S|} \sum_{g\in S} g -  e$$ in $ C_{\max}^{*}(G)$.
 Taking $S$  a finite symmetric generating set and using the fact  that $G$ has property $T$, the operator $B_{S}$ has zero as an isolated eigenvalue with eigenprojection $P$ (called the Kazhdan projection)  in $C^{*}_{\max}(G)$.
   Then $\tilde Q:=\sigma(P)$ is a projection in $    \prod_{m\in \nat} \bbK(\bar V_{m})\otimes_{\max}  \prod_{n\in \nat } \bbK(V_{n})$.
 The projection $(p_{m}\otimes_{\max} p_{n})( \tilde Q)$ is the eigenprojection of the selfadjoint operator 
 $$(p_{m}\otimes_{\max} p_{n})(\sigma(B_{S}))=\frac{1}{|S|} \sum_{g\in S} \bar\rho_{m}(g)\otimes \rho_{n}(g) -  \id_{\bar V_{m}\otimes V_{n}}$$ on $\bar V_{m}\otimes V_{n}$ to the eigenvalue $0$. The latter  is the projection onto the $G$-invariant vectors.
 By Schur's Lemma, the subspace of these invariants vanishes for $m\not=n$, and is given by 
    the multiples of the  identity $\id_{V_{m}}$ under the identification
 $\bar V_{m}\otimes V_{n}\cong \Hom_{\C}(V_{m},V_{n})$. 
 Consequently, $(p_{m}\otimes_{\max} p_{n})( \tilde Q)$ vanishes for $m\not=n$ and is a one-dimensional projection if $m=n$.

%
%
%
%
%
%

We consider the Hilbert space $H:=\bigoplus_{\nat }L^{2}(\nat)$. Then
$\prod_{\nat}\bbK$ can be identified with the subalgebra of $B(H)$ of operators which
are diagonal for the sum decomposition and for which every diagonal block is compact.
For every $m$ in $\nat${,} we choose embeddings $\bar V_{m} \to L^{2}(\nat)$ and $V_{m}\to L^{2}(\nat)$. We {obtain} induced
embeddings $\bbK(\bar V_{m})\to \bbK$ and $\bbK(V_{m})\to \bbK$. Forming their product over $\nat$ and the maximal tensor product
we {obtain} a homomorphism
$$ \prod_{m\in \nat} \bbK(\bar V_{m})\otimes_{\max}  \prod_{n\in \nat } \bbK(V_{n})\to \prod_{\nat} \bbK \otimes_{\max} \prod_{\nat} \bbK\ .$$ We define the projection
 $Q$ as the image of $\tilde Q$  under this map.
 Its $K$-theory class $q$ then has the desired properties. 
 \end{proof}

\section{$K$-theory does not preserve compact maps}\label{koprtzehrthrteeh}
In this section, we explain a surprising consequence of the existence of the class described in  \cref{jzlptzrjtzhztrhr} about compact maps in $\EE$. 

{Recall that a map $f:X\to Y$ in a  presentable $\infty$-category $\cC$ is called compact if, for every filtered   diagram $Z:I\to \cC$, there exists the dashed arrow fitting into
\[
\xymatrix{
  \colim_{i \in I} \Map_{\mathcal{C}}(Y, Z_i) \ar[r] \ar[d] 
  & \Map_{\mathcal{C}}(Y, \colim_{i \in I} Z_i) \ar[d] \ar@{-->}[dl] \\
  \colim_{i \in I} \Map_{\mathcal{C}}(X, Z_i) \ar[r] 
  & \Map_{\mathcal{C}}(X, \colim_{i \in I} Z_i)
}\ .
\]}By \cite{budu} we know that the category $\EE$ is dualizable in $\Pr^{\mathrm{L}}_{\st}$. Equivalently,  every $\aleph_1$-compact object $X$ in $\EE$ admits an exhaustion
\[ 
	\colim \left( X_0 \to X_1 \to X_2 \to \dots \right) \simeq X
\]
such that {all} the transition maps are compact. Since $\EE_{\sepa}$ is idempotent-complete and $\EE \simeq \Ind_{\aleph_{1}}(\EE_{\sepa})$ by \cref{hrteoptrhertg1} and \cref{okpherthertgertg}, the $\aleph_1$-compact objects in $\EE$ are given by $\EE_{\sepa}$. 
{Since $\ee(\C)$ is a compact object, the functor $K^{E}$ preserves colimits and is therefore a morphism in $\Pr^{L}_{\st}$. The following shows that it is not a morphism of dualizable stable  presentable $\infty$-categories.} 
\begin{theorem}\label{opjwefgioherjgkljsdfg}
	The functor $K^E: \EE \to {\Mod_{KU}(\Sp)}$ does not preserve compact maps.
\end{theorem}\begin{proof} 
{We will   use of the following properties of compact maps:}
\begin{enumerate}
	\item \label{bfbwoeoaghhhjqh} Let $\bA$ be a presentably symmetric monoidal category, {that} is dualizable in $\Pr^{\mathrm{L}}_{\st}$ and such that the tensor unit $1$ is compact. Let $A,B$ be objects in $\bA$ and let $ A^\vee = \hom_{\bA}(A,1) $ be the {internal} dual of $A$ in $\bA$. We consider the composite
	\[ 
		\mathrm{tr}: \map_{\bA}(1, A^\vee \otimes B) \xrightarrow{A \otimes -} \map_{\bA}(A, A \otimes A^\vee \otimes B) \xrightarrow{(\ev \otimes \id)_\ast} \map_{\bA}(A, B) \ ,
	\]
	where $A \otimes A^\vee \to 1$ is the evaluation. The image of the above composite (on the level of $\pi_0$) are the so-called trace-class maps. In the given situation, all trace-class maps are compact \cite[Lemma 3.9]{Ramzi_2026}.
	\item\label{vioapbnipaqghjw} Let $\bA$ a compactly generated category and let $f: A \to B$ be a compact map in $\bA$. Then there exists a compact object $T$ and factorization
	\[
	\xymatrix{
		A \ar[rr]^-{f} \ar[dr] & & B \\
		& T \ar[ur] & & \ .
	}
	\]
\end{enumerate}
 
	We first show that for a compact map of $KU$-modules $f: X \to Y$, the image of the induced map $\pi_0 f: \pi_0 X \to \pi_0 Y$ is finitely generated. By \ref{vioapbnipaqghjw}. it suffices to show that for all compact $KU$-modules $T$ the abelian groups $\pi_\ast T $ are finitely generated. This follows immediately from the observation that the groups $\pi_\ast KU$ are finitely generated and that the class of $KU$-modules $S$ such that $\pi_\ast S$ is finitely generated is closed under retracts, cofibers, and loops.
	
	Next, we show that the canonical map 
	$$i: \bigoplus_\nat \ee(\C) \to \prod_\nat \ee(\C)  $$
	is compact. In more detail, $i$ is the map corresponding under the identification
	\[ 
		\pi_0 \map_\EE(\bigoplus_\nat \ee(\C), \prod_\nat \ee(\C))\simeq \pi_0 \prod_{\nat \times \nat}  KU \simeq \prod_{\nat \times \nat} \Z
	\]
	to the class $(q_{m,n})_{(m,n) \in \nat \times \nat}$ satisfying
	$$
		q_{m,n}=\left\{\begin{array}{cc} 1 & m=n\\ 0  &  m\not=n \end{array} \right.\ .
	$$
	Specializing the map $\mathrm{tr}$ from {\ref{bfbwoeoaghhhjqh}.} to $\bA = \EE$, $A = \bigoplus_\nat \ee(\C)$, and $B = \prod_\nat \ee(\C)$, we obtain the map
	\begin{equation}\label{equ: webfghwqaghsg}
		\begin{split}
				K^E(\prod_\nat \ee(\C) \otimes_{\max} \prod_\nat \ee(\C))  = & \map_{\EE}({\ee(\C)}, \prod_\nat \ee(\C) \otimes_{\max} \prod_\nat \ee(\C)) \\
				\xrightarrow{\mathrm{tr}} & 	\map_\EE(\bigoplus_\nat \ee(\C), \prod_\nat \ee(\C)) \\
				\simeq & \prod_{\nat \times \nat}  KU \ .
		\end{split}
	\end{equation}
		On $\pi_0${,} the above composite agrees with the map in \eqref{gewrfwerfwergwerg}. Consequently, \cref{jzlptzrjtzhztrhr} implies that there exists a preimage $q$ of $(q_{m,n})_{(m,n) \in \nat \times \nat}$ under the composite in \eqref{equ: webfghwqaghsg}. This implies that $i$ is trace-class  {and therefore compact.}
		
		Lastly, we show that $K^E(i)$ is not compact. Indeed, $K^E(i)$ is given by the canonical map
		\[ 
			j: \bigoplus_\nat KU \to \prod_\nat KU
		\]
		because $K^E$ preserves coproducts and products. Since the image of $\pi_0(j)$ is not finitely generated, $j$ is clearly not compact. Therefore, $K^E(i) \simeq j$ is not compact.
\end{proof}

\begin{rem} 
	In the above proof, we have established that the canonical morphism
	\[
	i\colon \bigoplus_{\nat} \ee(\C) \longrightarrow \prod_{\nat} \ee(\C)
	\]
	in $\EE$ is compact. This behavior is unexpected because $i$ is a pure monomorphism in the sense of \cite[Appendix E]{Efimov:2024aa}, meaning that the  connecting map
	\[
	\Fib(i) \longrightarrow \bigoplus_{\nat} \ee(\C)
	\]
	is a phantom morphism. {
	Knowing that a compact split monomorphism has a compact domain,
	 one might naturally expect  that any compact pure monomorphism $f\colon A \to B$ in a dualizable category $\bA$ still has a compact domain $A$. However, since   $\bigoplus_{\nat} \ee(\C)$ is not compact, the morphism $i$ provides a direct counterexample to this expectation.} \hB
\end{rem}

    \bibliographystyle{alpha}
\bibliography{forschung2021}

\begin{thebibliography}{RSW25}

\bibitem[BD24]{budu}
U.~Bunke and B.~Duenzinger.
\newblock ${E}$-theory is compactly assembled.
\newblock
  \href{https://arxiv.org/pdf/2402.18228.pdf}{https://arxiv.org/pdf/2402.18228.pdf},
  2024.

\bibitem[BE20]{buen}
U.~Bunke and A.~Engel.
\newblock {\em Homotopy theory with bornological coarse spaces}, volume 2269 of
  {\em Lecture Notes in Math.}
\newblock Springer, 2020.
\newblock \href{https://arxiv.org/abs/1607.03657}{arXiv:1607.03657}.

\bibitem[BEL]{KKG}
U.~Bunke, A.~Engel, and M.~Land.
\newblock A stable $\infty$-category for equivariant $\mathrm{K\!K}$-theory.
\newblock \href{https://arxiv.org/pdf/2102.13372.pdf}{arxiv:2102.13372}.

\bibitem[Bun24]{bunke:2023aa}
U.~Bunke.
\newblock {{\(KK\)}}- and {{\(E\)}}-theory via homotopy theory.
\newblock {\em Orbita Math.}, 1(2):103--210, 2024.

\bibitem[CH90]{zbMATH04182148}
A.~Connes and N.~Higson.
\newblock Deformations, asymptotic morphisms and bivariant {{\(K\)}}-theory.
\newblock {\em C. R. Acad. Sci., Paris, S{\'e}r. I}, 311(2):101--106, 1990.

\bibitem[Cis19]{Cisinski:2017}
D.-Ch. Cisinski.
\newblock {\em Higher categories and homotopical algebra}, volume 180 of {\em
  Cambridge studies in advanced mathematics}.
\newblock Cambridge University Press, 2019.
\newblock \url{http://www.mathematik.uni-regensburg.de/cisinski/CatLR.pdf}.

\bibitem[Cun87]{MR899916}
Joachim Cuntz.
\newblock A new look at {$KK$}-theory.
\newblock {\em $K$-Theory}, 1(1):31--51, 1987.

\bibitem[DJ25]{Datta:2025aa}
A.~Datta and M.~Joachim.
\newblock Equivariant $kk$-theory and model categories.
\newblock
  \href{https://arxiv.org/pdf/2506.16238.pdf}{https://arxiv.org/pdf/2506.16238.pdf},
  06 2025.

\bibitem[Efi24]{Efimov:2024aa}
A.~I. Efimov.
\newblock K-theory and localizing invariants of large categories.
\newblock
  \href{https://arxiv.org/pdf/2405.12169.pdf}{https://arxiv.org/pdf/2405.12169.pdf},
  05 2024.

\bibitem[Efi25]{Efimov_2025}
Alexander~I. Efimov.
\newblock Rigidity of the category of localizing motives, 2025.

\bibitem[GHT00]{Guentner_2000}
E.~Guentner, N.~Higson, and J.~Trout.
\newblock Equivariant ${E}$-theory for ${C}^{*}$-algebras.
\newblock {\em Memoirs of the American Mathematical Society}, 148(703):0--0,
  2000.

\bibitem[Hig87]{higsondiss}
N.~Higson.
\newblock {A characterization of {KK}-theory}.
\newblock {\em Pacific J.\ Math.}, 126:253--276, 1987.

\bibitem[Hig90a]{MR1068250}
N.~Higson.
\newblock Categories of fractions and excision in {$KK$}-theory.
\newblock {\em J. Pure Appl. Algebra}, 65(2):119--138, 1990.

\bibitem[Hig90b]{higson}
N.~Higson.
\newblock Categories of fractions and excision in {$KK$}-theory.
\newblock {\em J. Pure Appl. Algebra}, 65(2):119--138, 1990.

\bibitem[Kas88]{kasparovinvent}
G.~G. Kasparov.
\newblock {Equivariant $K\!K$-theory and the Novikov conjecture}.
\newblock {\em Invent.\ Math.}, 91(1):147--201, 1988.

\bibitem[LN18]{LN}
M.~Land and T.~Nikolaus.
\newblock {On the relation between K- and L-theory of $C^*$-algebras}.
\newblock {\em Math. Ann.}, 371:517--563, 2018.

\bibitem[Lur26]{kerodon}
Jacob Lurie.
\newblock Kerodon.
\newblock \url{https://kerodon.net}, 2026.

\bibitem[MN06]{MR2193334}
R.~Meyer and R.~Nest.
\newblock The {B}aum--{C}onnes conjecture via localisation of categories.
\newblock {\em Topology}, 45(2):209--259, 2006.

\bibitem[Oza03]{Ozawa_2003}
N.~Ozawa.
\newblock An application of expanders to $b(\ell^{2})\otimes b(\ell^{2})$.
\newblock {\em Journal of Functional Analysis}, 198(2):499--510, March 2003.

\bibitem[Ram26]{Ramzi_2026}
M.~Ramzi.
\newblock Locally rigid $\infty$-categories.
\newblock
  \href{https://arxiv.org/abs/2410.21524}{https://arxiv.org/abs/2410.21524},
  2026.

\bibitem[RSW25]{RSW}
M.~Ramzi, V.~Sosnilo, and C.~Winges.
\newblock Every motive is the motive of a stable $\infty$-category.
\newblock
  \href{https://arxiv.org/abs/2503.11338}{https://arxiv.org/abs/2503.11338},
  2025.

\bibitem[Uuy13]{Uuye:2010aa}
O.~Uuye.
\newblock {Homotopy algebra for $C^{*}$-algebras}.
\newblock {\em J.\ Noncommut.\ Geom.}, 7(4):981--1006, 2013.

\bibitem[WY12]{Willett_2012}
R.~Willett and G.~Yu.
\newblock Higher index theory for certain expanders and {G}romov monster
  groups, i.
\newblock {\em Advances in Mathematics}, 229(3):1380--1416, February 2012.

\end{thebibliography}

\end{document}